\documentclass[a4paper, 12pt]{amsart}
\usepackage{amsmath,amssymb,amsthm}
\usepackage{amscd}
\usepackage{array,enumerate}
\newtheorem{Thm}{Theorem}[section]
\newtheorem{Prop}[Thm]{Proposition}
\newtheorem{Lem}[Thm]{Lemma}
\newtheorem{Cor}[Thm]{Corollary}
\newtheorem{Thmint}{Theorem}[section]

\newtheorem{Defint}[Thmint]{Definition}
\newtheorem{MThm}{Main}

\theoremstyle{definition}
\newtheorem{Rem}[Thm]{Remark}

\newtheorem{Exam}[Thmint]{Example}

\newcommand{\Cs}{C$^\ast$}

\newcommand{\id}{\mbox{\rm id}}

\newcommand{\rg}{\mathop{{\mathrm C}_{\mathrm r}^\ast}}
\newcommand{\fg}{\mathop{{\mathrm C}^\ast}}
\newcommand{\rc}{\mathop{\rtimes _{\mathrm r}}}

\newcommand{\rca}[1]{\mathop{\rtimes _{{\mathrm r}, #1}}}
\newcommand{\Cn}[1]{\mathcal{O}_{#1}}

\newcommand{\U}{\mathcal{U}}
\DeclareMathOperator{\supp}{supp}

\DeclareMathOperator{\bigfp}{\lower0.25ex\hbox{\LARGE $\ast$}}
\newcommand{\ad}{\mathop{\rm Ad}}
\title[Equivariant $\Cn{2}$-absorption for exact groups]
{Equivariant $\Cn{2}$-absorption theorem for exact groups}
\author{Yuhei Suzuki}
\address{Department of Mathematics, Faculty of Science, Hokkaido University,
Kita 10, Nishi 8, Kita-Ku, Sapporo, Hokkaido, 060-0810, Japan}
\email{yuhei@math.sci.hokudai.ac.jp}
\begin{document}
\maketitle
\begin{abstract}
We show that, up to strong cocycle conjugacy,
every countable exact group admits a unique equivariantly $\Cn{2}$-absorbing,
pointwise outer action on the Cuntz algebra $\Cn{2}$ with the quasi-central approximation property (QAP).
In particular, we establish the equivariant analogue of
the Kirchberg $\Cn{2}$-absorption theorem for these groups.
\end{abstract}
\section{Introduction}
The Kirchberg $\Cn{2}$-absorption theorem (\cite{Kir}, \cite{KP}) states that the Cuntz algebra $\Cn{2}$ tensorially absorbs all unital simple separable nuclear \Cs-algebras $A$;
$A \otimes \Cn{2} \cong \Cn{2}$.
This is one of the fundamental ingredients of Phillips's proof of the
 classification theorem of Kirchberg algebras \cite{Phi}
(see \cite{Kir} for Kirchberg's  original approach).
Here recall that a \Cs-algebra is said to be a Kirchberg algebra
if it is purely infinite simple, separable, nuclear.
We refer the reader to the book \cite{Ror} for basic facts and background on
the Kirchberg algebras.

Considering the special roles of this beautiful, deep, surprising theorem of Kirchberg,
and also observing the recent developments on the classification of \Cs-dynamical systems
(see e.g., \cite{IM}, \cite{IM2}, \cite{Sz18}),
it is a natural and important attempt to establish an equivariant analogue of the $\Cn{2}$-absorption theorem.
(For history and background of the study of group actions on operator algebras,
we refer the reader to \cite{IICM} or the introduction of \cite{Sz18} and references therein.)
Indeed, among other things, recently Szab\'o \cite{Sz18}
established the theorem for countable amenable groups.
Here we extend the theorem to countable exact groups.
The precise statement is as follows.
\begin{MThm}
Let $G$ be a countable exact group.
Let $\delta \colon G \curvearrowright \Cn{2}$ be
an equivariantly $\Cn{2}$-absorbing, pointwise outer
action with the quasi-central approximation property $($QAP$)$ \cite{BEW}.
Then for any action $\alpha \colon G \curvearrowright A$
on a unital simple separable nuclear \Cs-algebra,
the diagonal action $\alpha \otimes \delta$ is strongly cocycle conjugate to $\delta$.
\end{MThm}
For the definition of the QAP, see Sections \ref{subsection:l2} and \ref{subsection:qap}.
Note that such a $\delta$ always exists (see Example \ref{Exam:model} below or \cite{Suzeq}),
and (therefore) all the assumptions cannot be removed (cf.~Remark \ref{Rem:abs}).
Thus one can regard this theorem as \emph{the} equivariant $\Cn{2}$-absorption theorem.

As a rephrasing of the Main Theorem, we obtain the following uniqueness theorem.
\begin{Thmint}\label{Thmint:unique}
Let $G$ be a countable exact group.
Then, up to strong cocycle conjugacy,
there is a unique equivariantly $\Cn{2}$-absorbing, pointwise outer, QAP
action $G\curvearrowright \Cn{2}$.
\end{Thmint}
It is notable that, as all the three conditions are stable under cocycle conjugacy, the above theorem gives the first abstract
characterization of a (strong) cocycle conjugacy class of pointwise outer actions
of a non-amenable group on a simple operator algebra.

We emphasize that all the previously known classification results of
simple equivariant
operator algebras (in both \Cs- and W$^\ast$-cases) use amenability 
of the acting group in substantial ways.
Theorem \ref{Thmint:unique} is the first positive evidence that a fruitful classification theory of \Cs-dynamical systems could be developed even beyond amenable groups.

\begin{Exam}\label{Exam:model}
Here we give motivating examples of actions which satisfy the assumptions of the Main Theorem; cf.~\cite{Suzeq}.
Let $G$ be a countable exact group.
For each $n\in \mathbb{N}$,
choose a (topologically) amenable action $\eta_n \colon G \curvearrowright X_n$
on a compact metrizable space $X_n$ (\cite{BO}, Theorem 5.1.7).
For each $n\in \mathbb{N}$,
choose a unital embedding $\iota_n \colon C(X_n) \rca{\eta_n} G \rightarrow \Cn{2}$ \cite{KP}, \cite{Kir}.
Each $\iota_n$ defines a unitary representation
$u_n \colon G \rightarrow \mathcal{U}(\Cn{2})$ by the restriction.
Then the diagonal action \[\bigotimes _{n \in \mathbb{N}} \ad(u_n) \colon G \curvearrowright \bigotimes _{n \in \mathbb{N}} \Cn{2}\]
defines an action \[\delta \colon G \curvearrowright \Cn{2}\]
via an isomorphism $\bigotimes _{n \in \mathbb{N}} \Cn{2} \cong \Cn{2}$ (\cite{Ror}).
It is clear from the construction
that $\delta$ satisfies the assumptions of the Main Theorem.
Thus, as a particular consequence of the Main Theorem,
the strong cocycle conjugacy class of
$\delta$ is independent on the choices of $\eta_n$ and $\iota_n$.
This already implies the following non-trivial consequence:
For any amenable action $\eta \colon G \curvearrowright X$
on a compact metrizable space
and for any choices of $\eta_n$ and $\iota_n$,
there is a unital $G$-equivariant embedding
\[(C(X), \eta) \rightarrow ((\Cn{2})_\omega, \delta_\omega).\]

\end{Exam}

Our proof of the Main Theorem mostly follows that of Szab\'o's theorem (\cite{Sz18}, Theorem C).
However at a few crucial steps, amenability of the acting group is used in essential ways.
In particular, this includes the most important step Lemma 2.3 of \cite{Sz18}.
This lemma is used to carry out Connes's two-by-two matrix trick \cite{Con}
(see \cite{Sz18}, Lemmas 2.4 and 5.1).
Unfortunately the statement of the lemma no longer holds true for non-amenable groups; see Remark \ref{Rem:fp}.
However, we will see in Lemma \ref{Lem:fp} that, under appropriate assumptions,
the statement is still valid.
This is the main contribution of the present article toward the Main Theorem.

As an intermediate result in the proof of the Main Theorem,
we also obtain the following useful characterizations of the QAP in an important case.
We use this result together with Lemma \ref{Lem:fp} to deduce the key result Corollary \ref{Cor:pi}.

\begin{Thmint}\label{Thmint:QAP}
Let $\alpha \colon G \curvearrowright A$ be an action of a countable exact group
on a unital simple separable nuclear \Cs-algebra.
Let $\omega$ be a free ultrafilter on $\mathbb{N}$.
Then the following conditions are pairwise equivalent.
\begin{enumerate}[\upshape(1)]
\item The $\alpha$ has the QAP.
\item There is a sequence $(\xi_n)_{n=1}^\infty$ in $\ell^2(G, A_\omega)$
satisfying
$\lim_{n \rightarrow \infty} \langle \xi_n, \tilde{\alpha}_g(\xi_n)\rangle=1$ for all $g\in G$.
\item There is a sequence $(k_n \colon G \rightarrow A_\omega)_{n=1}^\infty$
of finitely supported positive definite functions
satisfying $\lim_{n\rightarrow \infty} k_n(g) =1$ for all $g\in G$.
\item The induced action $\alpha^\omega \colon G \curvearrowright A^\omega$ on the ultrapower $A^\omega$ has the QAP.

\item The induced action $\alpha_\omega \colon G \curvearrowright A_\omega$
on the central sequence algebra $A_\omega$ has the QAP.
\item For any separable $G$-\Cs-subalgebra $B\subset A^\omega$,
the restriction of $\alpha^\omega$ to $A^\omega \cap B'$ has the QAP.
\end{enumerate}
\end{Thmint}
We stress that even for some actions in Example \ref{Exam:model} (see also \cite{Suzeq}, \cite{Suz19}),
validity of conditions (2) to (6) are not obvious.
The crucial implication used in the present article
is (1) $\Rightarrow$ (4); cf.~Lemma \ref{Lem:fp} and Corollary \ref{Cor:pi}.

We also discuss an equivariant analogue of the Kirchberg $\Cn{\infty}$-absorption theorem \cite{Kir}, \cite{KP}.
Recall that the theorem
states that the Cuntz algebra $\Cn{\infty}$ is tensorially absorbed by
any Kirchberg algebra $A$; $A\otimes \Cn{\infty} \cong A$.
This is another main ingredient of the proof of the Kirchberg--Phillips classification theorem in \cite{Phi}. Therefore, similar to the $\Cn{2}$-absorption theorem,
it is desirable to obtain an equivariant analogue of the $\Cn{\infty}$-absorption theorem.
Indeed Szab\'o obtains such results
for countable amenable groups (see Theorems B and 3.5 in \cite{Sz18}).
In this article, we prove the following result for QAP actions on unital Kirchberg algebras.

We first give model (absorbed) actions on $\Cn{\infty}$.
\begin{Defint}\label{Defint:fact}
Let $u \colon G \rightarrow \mathcal{U}(\Cn{\infty})$
be a unitary representation of a group $G$.
We say that \emph{$u$ factors through $\Cn{2}$}
if there is a unitary representation $v \colon G \rightarrow \mathcal{U}(\Cn{2})$
and a $($non-unital$)$ embedding $\iota \colon \Cn{2}\rightarrow \Cn{\infty}$ satisfying 
\[u_g= \iota(v_g)+ (1_{\Cn{\infty}}-\iota(1_{\Cn{2}})) \qquad{\rm~ for~ all~} g\in G.\]
{\rm Note that by the Kirchberg $\Cn{2}$-embedding theorem $($applied to the reduced group \Cs-algebra$)$, every countable exact group admits
such a faithful representation.
Note that unlike Example 3.6 in \cite{Sz18}, here we do not assume the injectivity of
the induced $\ast$-homomorphism on $\fg(G)$.}
\end{Defint}
\begin{Exam}[cf.~Example 3.6 in \cite{Sz18}]\label{Exam:model2}
Let $G$ be a countable group.
Let $(u_n)_{n=1}^\infty$ be a sequence of unitary representations of $G$ on $\Cn{\infty}$ which factor through $\Cn{2}$.
Then the diagonal action
\[\bigotimes_{n\in \mathbb{N}} \ad(u_n)\colon G \curvearrowright \bigotimes_{\mathbb{N}} \Cn{\infty}\]
defines an action
\[\gamma \colon G \curvearrowright \Cn{\infty}\]
via an isomorphism $\bigotimes_{\mathbb{N}} \Cn{\infty} \cong \Cn{\infty}$ (\cite{Ror}).
We remark that $\gamma$ admits an invariant state.
Thus, when $G$ is non-amenable, $\gamma$ cannot have the QAP.
\end{Exam}
\begin{Thmint}\label{Thmint:Oinf}
Let $G$ be a countable exact group.
Let $\alpha \colon G \curvearrowright A$
be a pointwise outer, QAP action on a unital Kirchberg algebra.
Then for any action $\gamma$ in Example \ref{Exam:model2},
$\alpha$ is strongly cocycle conjugate to $\alpha \otimes \gamma$.
In particular $\alpha$ is equivariantly $\Cn{\infty}$-absorbing.
\end{Thmint}

\section{Preliminaries}
Here we recall some definitions and facts, and fix notations.

\subsection{Hilbert \Cs-bimodule $\ell^2(G, A)$}\label{subsection:l2}
Let $G$ be a discrete group. Let $A$ be a \Cs-algebra.
Define the linear space
\[\ell^2(G, A) := \left\{ (\xi_g)_{g\in G} \in \prod_G A: \sum_{g\in G} \xi_g^\ast \xi _g{\rm~converges~ in~ norm}\right\}.\]
For $a\in A$ and $\xi \in \ell^2(G, A)$,
define $a\xi, \xi a \in \ell^2(G, A)$ to be
\[(a\xi)_g := a \xi_g,\quad (\xi a)_g:=\xi_g a\qquad{\rm for~} g\in G.\]
For $\xi, \eta \in \ell^2(G, A)$, set
\[\langle \xi, \eta\rangle := \sum_{g\in G} \xi_g^\ast \eta_g \in A.\]
Note that these operations make $\ell^2(G, A)$ a Hilbert \Cs-bimodule over $A$. For $\xi \in \ell^2(G, A)$, set \[\|\xi \|:= \| \langle \xi, \xi\rangle\|^{1/2}.\]
This
defines a complete norm on $\ell^2(G, A)$.

Let $\alpha \colon G \curvearrowright A$ be an action.
Then define the norm-preserving action $\tilde{\alpha} \colon G \curvearrowright \ell^2(G, A)$
to be
\[\tilde{\alpha}_g(\xi)_h := \alpha_g(\xi_{g^{-1} h}) \qquad{\rm~ for~} g, h\in G, \xi \in \ell^2(G, A).\]
This action is used to formulate amenability-type conditions for \Cs-dynamical systems; see \cite{AD}, \cite{AD02}, \cite{BO}, \cite{BEW} for instance.
\subsection{Amenable actions on \Cs-algebras}\label{subsection:qap}
Recently, in \cite{Suzeq}, we discovered
actions of non-amenable groups on simple \Cs-algebras with properties which should be
regarded as amenability of \Cs-dynamical systems.
A few applications of such actions are already found in \Cs-algebra theory; see \cite{Suz19}, \cite{Suz20a}, \cite{Suz20b}.
We believe that this novel phenomenon provides a new rich field in
\Cs-algebra theory. One of the motivations behind the present
article is to examine a new evidence supporting this my expectation from a different angle.

While the definitive definition of amenability is still not known
for (non-commutative) \Cs-dynamical systems, in \cite{BEW}, motivated by examples constructed in \cite{Suzeq},
a new amenability-type condition, called the \emph{quasi-central approximation property} (\emph{QAP}), is introduced for \Cs-dynamical systems.
(We remark that the situation is totally different in the von Neumann algebra setting.
In the realm of von Neumann algebras, the definitive definition is already known,
and non-amenable discrete groups cannot act amenably on a factor \cite{AD79}.) 
For simplicity, here we recall the definition of the QAP only in the unital and discrete case.
See \cite{BEW} for the general case.
Recall that a discrete group action $\alpha \colon G \curvearrowright A$
on a unital \Cs-algebra is said to have the \emph{quasi-central approximation property} (\emph{QAP})
if there is a net $(\xi_i)_{i\in I}$ in $\ell^2(G, A)$
satisfying
\begin{enumerate}
\item $\lim_{i\in I} \langle \xi_i, \tilde{\alpha}_g (\xi_i) \rangle = 1_A$ for all $g\in G$,
\item $\lim_{i\in I} \|a \xi_i - \xi_i a\|=0$ for all $a\in A$.
\end{enumerate}
For actions on commutative \Cs-algebras, it is not hard to see that
the QAP is equivalent to the (topological) amenability (cf.~\cite{BO}, Lemma 4.3.7).
It is also clear that the QAP
is stable under taking equivariant inductive limits.
In particular, the actions constructed in Example \ref{Exam:model}
have the QAP.

In Theorem \ref{Thmint:QAP}, for some important cases, we obtain
useful characterizations of the QAP in terms of the (relative) central sequence algebra. This is another important step to prove the Main Theorem; cf.~Corollary \ref{Cor:pi}.

Here we record the following basic observation on the QAP. The proof is
easy and straightforward
and we leave details to the reader.
\begin{Prop}
Let $\alpha$ and $\beta$ be an action of a discrete group $G$ on a unital \Cs-algebra $A$ such that $\alpha_g \circ \beta_g^{-1}$ is inner for all $g\in G$.
Then $\alpha$ has the QAP if and only if so does $\beta$.
In particular, the $QAP$ is stable under cocycle conjugacy.
\end{Prop}
Another object used to formulate amenability-type conditions of actions
is positive definite functions; see e.g., \cite{AD}, \cite{AD02}, \cite{BO}, \cite{BEW}.
Recall that for a \Cs-dynamical system $\alpha \colon G \curvearrowright A$,
a function $f \colon G \rightarrow A$ is said to be \emph{positive definite} \cite{AD}
if it satisfies the following condition:
For any finite sequence $g_1, \ldots, g_n \in G$,
the matrix $(\alpha_{g_i}(f(g_i^{-1} g_j)))_{1\leq i, j \leq n} \in \mathbb{M}_n(A)$ is positive.
\subsection{Ultrapowers and central sequence algebras}\label{Sub:omega}
Throughout this article, we fix a free ultrafilter $\omega$ on $\mathbb{N}$.
For a \Cs-algebra $A$, the \emph{ultrapower} $A^\omega$ of $A$ with respect to $\omega$ is defined to be
\[ A^\omega := \ell^\infty(\mathbb{N}, A) / \left\{(x_n)_{n=1}^\infty \in  \ell^\infty(\mathbb{N}, A) : \lim_{n\rightarrow \omega} \| x_n\|=0\right\}.\]
[We alert that some authors (including \cite{Sz18}) use the symbol $A_\omega$ for the \Cs-algebraic ultrapower.
In this article, following e.g., \cite{IM}, \cite{IM2}, we employ the notation
$A^\omega$ which is standard at least in von Neumann algebra theory.]
For $(x_n)_{n=1}^\infty \in \ell^\infty(\mathbb{N}, A)$, denote by
$[x_n]_n^\omega$ its image in $A^\omega$.

We regard $A$ as a \Cs-subalgebra of $A^\omega$ via the embedding
$a \mapsto [(a, a, a, \ldots)]_n^\omega$.
The relative commutant \Cs-algebra
\[A_\omega:=A^\omega \cap A'\]
 is said to be the \emph{central sequence algebra} of $A$ (with respect to $\omega$).

For any discrete group action $\alpha \colon G \curvearrowright A$,
we have the actions
$\alpha ^\omega \colon G \curvearrowright A^\omega$, $\alpha_\omega \colon G \curvearrowright A_\omega$
inherited from the pointwise application of $\alpha$ to $\ell^\infty(\mathbb{N}, A)$.
When the underlying \Cs-algebra is obvious from the context, we
simply denote $\alpha^\omega$, $\alpha_\omega$, and their restrictions
to $G$-\Cs-subalgebras (typically relative commutant \Cs-subalgebras) by $\alpha$ for short.
\subsection{Exact groups}
Recall that a discrete group $G$ is said to be \emph{exact} \cite{KW}
if the reduced group \Cs-algebra $\rg(G)$ is exact.
We refer the reader to \cite{BO} and \cite{OzICM} for introduction and applications of this rich property.
Here we recall only a few things. First, it should be stressed that, comparing to amenable groups, the class of exact groups is fairly huge.
For instance the class includes all hyperbolic groups and linear groups.
Moreover the class is stable under
taking extensions, subgroups, increasing unions, amalgamated free products.
All the currently known non-exact groups are constructed
by a probabilistic method \cite{Gro}, \cite{Os}.
By Ozawa's theorem \cite{Oz},
exactness of a countable group is equivalent to
the existence of a (topologically) amenable action on a compact metrizable space (see also \cite{BO}, Theorem 5.1.7).
\subsection{(Strong) cocycle conjugacy}
Let $\alpha \colon G \curvearrowright A$ be a discrete group action on a unital \Cs-algebra.
A map $u \colon G \rightarrow \U(A)$ is said to be an \emph{$\alpha$-cocycle}
if it satisfies $u_{gh}=u_g \alpha_g(u_h)$ for all $g, h\in G$.
Each $\alpha$-cocycle $u$ defines
a new action $\alpha^u \colon G \curvearrowright A$, called the \emph{cocycle perturbation} of $\alpha$ by $u$,
by the formula $\alpha^u_g:= \ad(u_g) \circ \alpha_g$ for $g\in G$.
Two actions $\alpha \colon G \curvearrowright A$ and $\beta \colon G \curvearrowright B$ are said to be \emph{cocycle conjugate}
if $\alpha$ is conjugate to a cocycle perturbation of $\beta$,
that is, if there is a $\beta$-cocycle $w$ and an isomorphism $\varphi \colon B \rightarrow A$ satisfying
$\alpha_g = \varphi \circ \beta_g^w \circ \varphi^{-1}$ for all $g\in G$.
Clearly, if two actions $\alpha$ and $\beta$ are cocycle conjugate,
then $\alpha_\omega$ and $\beta_\omega$ are conjugate.

Two actions $\alpha$ and $\beta$ are said to be \emph{strongly cocycle conjugate}
if there is an isomorphism $\varphi \colon B \rightarrow A$, a $\beta$-cocycle $w$, and
a sequence $(u_n)_{n=1}^\infty$ in $\U(B)$
satisfying
\[\alpha_g = \varphi \circ \beta_g^w \circ \varphi^{-1} 
\quad {\rm and} \quad w_g = \lim_{n\rightarrow \infty} u_n \beta_g(u_n)^\ast
\qquad{\rm~for~ all~} g\in G.\]
For an advantage of strong cocycle conjugacy rather than just cocycle conjugacy,
see Lemma 4.2 of \cite{SzI} for instance.
Note that these two relations are reflexive, symmetric, and transitive.
\subsection{Saturated \Cs-dynamical systems}
The definition of saturated \Cs-dynamical systems
is introduced by Szab\'o \cite{Sz18} as an equivariant analogue of \cite{FH}.
Roughly speaking, this is an abstract formulation of some important properties
shared by ultraproduct actions and their restrictions to (relative) central sequence algebras. These actions are thus typical examples of saturated actions.
This concept allows us to considerably reduce the number of times to repeat \emph{standard reindexation arguments}.
We thus employ this concept in the article.
We refer the reader to Section 1 of \cite{Sz18} for
the definition and basic facts of saturated actions.
\subsection{$\mathcal{D}$-absorbing actions}
Let $\mathcal{D}$ be a strongly self-absorbing \Cs-algebra \cite{TW}.
(In this article we only consider the Cuntz algebra cases $\mathcal{D}=\Cn{2}, \Cn{\infty}$.)
Let $\alpha$ be an action of a discrete group $G$
on a unital \Cs-algebra $A$.
Recall that $\alpha$
is said to be \emph{equivariantly $\mathcal{D}$-absorbing}
if $\alpha$ is cocycle conjugate to $\id_{\mathcal{D}} \otimes \alpha \colon G \curvearrowright \mathcal{D} \otimes A$.
An equivariant McDuff-type theorem
(see \cite{IM}, Theorem 4.11 or \cite{SzI}, Corollary 3.8)
shows that $\alpha$ is equivariantly $\mathcal{D}$-absorbing
if and only if 
there is a unital embedding
$\mathcal{D} \rightarrow (A_\omega)^\alpha$.
Moreover, if this is the case, the theorem shows that $\alpha$
is in fact strongly cocycle conjugate to $\id_{\mathcal{D}} \otimes \alpha$.
\subsection{Notations}Here we fix basic notations used in this article.

Throughout the article, all \Cs-algebras are assumed to be nonzero. Let $A$, $B$ be unital \Cs-algebras. Let $X$ be a normed space. 
Throughout the article, $G$ denotes a countable discrete exact group.
\begin{itemize}
\item For $x, y \in X$ and $\epsilon>0$,
denote by $x \approx_\epsilon y$ if $\|x-y\|<\epsilon$.
\item For $x, y \in A$,
set $[x, y]:= xy -yx$.
\item Denote by $\U(A)$ the unitary group of $A$.
\item Denote by $1_A$ the unit of $A$.
When the considered \Cs-algebra $A$ is clear from the context,
denote $1_A$ by $1$ for short. 
\item For $u\in \U(A)$, $\ad(u)$ denotes the inner automorphism of $A$ defined by $u$;
$\ad(u)(x)=u x u^\ast$ for $x\in A$.
\item For a unitary representation $u \colon G \rightarrow \U(A)$,
$\ad(u)$ denotes the inner action $G \curvearrowright A$ given by
$g \mapsto \ad(u_g)$.
\item The symbols `$\otimes$', `$\rc$' denote the minimal tensor product
(of \Cs-algebras and completely positive maps)
and the reduced crossed product respectively.
\item For an action $\alpha \colon G \curvearrowright A$,
$A^\alpha$ denotes the fixed point algebra of $\alpha$;
\[A^\alpha:=\{a\in A: \alpha_g(a)=a {\rm~for~all~}g\in G\}.\]
\item For two actions $\alpha \colon G \curvearrowright A$ and $\beta \colon G \curvearrowright B$, denote by $\alpha \otimes \beta \colon G \curvearrowright A\otimes B$ the diagonal action; $(\alpha\otimes \beta)_g := \alpha_g \otimes \beta_g$ for $g\in G$.
\item For a subset $S \subset A$,
set $A \cap S' := \{ a\in A: [a, s]=0 {\rm~for~all~}s\in S\}$.
\item For $n\in \mathbb{N}$, denote by $\mathbb{M}_n(A)$
the \Cs-algebra of $n$-by-$n$ matrices over $A$.
When $A=\mathbb{C}$, we denote it by $\mathbb{M}_n$ for short.
We identify $\mathbb{M}_n(A)$ with $\mathbb{M}_n \otimes A$
via the canonical isomorphism.
\item For $n\in \mathbb{N}$ and $1\leq i, j \leq n$, denote by $e_{i, j} \in \mathbb{M}_n(A)$ the matrix whose $(i, j)$-entry is $1_A$ and the other entries are $0$. 
\item Let $\alpha \colon G \curvearrowright A$.
Let $B \subset A$ be a $G$-\Cs-subalgebra.
When there is no confusion, we denote
the restriction action $G \curvearrowright B$ by the same symbol $\alpha$.
\item Denote by $1_G$ the identity element of $G$.
\end{itemize}

\section{Proof of the Main Theorem}
The following key result is a generalization of \cite{Sz18}, Lemma 2.3.
Because F{\o}lner sets no longer exist in our case,
the statement involves technical assumptions (which may be regarded as a relative (weaker) version of the QAP).
In fact, the full statement of Lemma 2.3 fails for non-amenable groups;
see Remark \ref{Rem:fp}.
We will however see that the following statement is strong enough to obtain the Main Theorem.

Recall that $G$ denotes a countable exact group.
\begin{Lem}\label{Lem:fp}
Let $D$ be a unital \Cs-algebra.
Let $\alpha \colon G \curvearrowright D$ be a saturated action.
Let $A \subset D$ be a separable $G$-\Cs-subalgebra
satisfying the following conditions.
\begin{enumerate}[\upshape(1)]
\item The relative commutant $D \cap A'$ is purely infinite simple.
\item The restriction action $G \curvearrowright D \cap A'$ is pointwise outer.
\item 
For any finite subsets $S\subset A$, $F\subset G$ and any positive number $\epsilon>0$,
there is $\xi \in \ell^2(G, D)$ satisfying
\begin{itemize}
\item $\langle \xi, \tilde{\alpha}_g(\xi) \rangle \approx _\epsilon 1$ for all $g\in F$,
\item $a\xi \approx_\epsilon \xi a$ for all $a\in S$.
\end{itemize} 
\end{enumerate}
Then the fixed point algebra $(D \cap A') ^{\alpha}$ is purely infinite simple.
\end{Lem}
\begin{proof}
Let $a\in (D \cap A') ^{\alpha}$ be a positive element of norm one.
We will find, following Lemma 2.3 of \cite{Sz18},
a proper isometry $v\in  (D \cap A') ^{\alpha}$ 
satisfying $v^\ast a v =1$.

By Proposition 2.2 of \cite{Sz18},
we have a nonzero projection $p \in D \cap A'$ satisfying
\begin{itemize}
\item
$\alpha_g(p)  p=0$ for all $g\in G \setminus \{1_G\}$,
\item
$a\alpha_g(p)=\alpha_g(p)$ for all $g\in G$.
\end{itemize}
By replacing $p$ by its proper subprojection (which exists by (1)) if necessary,
we may assume in addition that there is a nonzero projection $q \in D \cap A'$ with
$q \alpha_g(p)=0$ for all $g\in G$.
Take an isometry $t \in D \cap A'$ with
$t^\ast p t=1$.

Let finite subsets $S \subset A$, $F \subset G$ and $\epsilon>0$ be given.
Choose $\xi \in \ell^2(G, D)$ as in (3) for the $S$, $F$, $\epsilon$.
By standard perturbation arguments, we may assume in addition that
$\xi$ is finitely supported and satisfies $\langle \xi, \xi \rangle =1$.
Set \[v_{S, F}:= \sum_{g\in G} \alpha_g(pt)\xi_g \in D.\]
Then direct computations imply
\[v_{S, F}^\ast v_{S, F}= \sum_{g\in G} \xi_g ^\ast \xi_g=1.\]
Observe also that $qv_{S, F} = 0$.
Hence $v_{S, F}$ is a proper isometry.
It is clear from the choice of $p$ that $av_{S, F}=v_{S, F}$.
Thus $v_{S, F}^\ast a v_{S, F}=1$.

We show that $v_{S, F}$ almost commutes with $S$
and is almost invariant under $F$
(the precise meanings are self-explaining below).
For $x\in S$, as $pt\in D \cap A'$,
we have
\[[x, v_{S, F}] =\sum_{g\in G} \alpha_g(pt) (x\xi_g - \xi_g x).\]
Hence
\[\|[x, v_{S, F}]\|^2=\|[x, v_{S, F}]^\ast [x, v_{S, F}]\|= \|\sum_{g\in G} (x\xi_g - \xi_g x)^\ast (x\xi_g - \xi_g x)\| =\| x\xi - \xi x\|^2 < \epsilon^2.\]
For $s\in F$,
we have
\[\alpha_s(v_{S, F})= \sum_{g\in G} \alpha_{sg}(pt) \alpha_s(\xi_g)
= \sum_{g\in G} \alpha_g(pt) \alpha_s(\xi_{s^{-1}g})
= \sum_{g\in G} \alpha_g(pt) \tilde{\alpha}_s(\xi)_g.\]
Therefore computations similar to the previous one yield
\[ \|v_{S, F}- \alpha_s(v_{S, F})\| = \| \xi - \tilde{\alpha}_s(\xi)\|<\sqrt{2\epsilon}.\]

Now, since $\alpha$ is saturated, $G$ is countable, and $A$ is separable, this guarantees the existence of a proper isometry $v \in (D\cap A')^\alpha$ with $v^\ast a v=1$.
\end{proof}
\begin{Rem}
We remark that the proof of Lemma 2.1 in \cite{Sz18}, which Proposition 2.2 of \cite{Sz18} (thus 
Lemma \ref{Lem:fp}) depends on, needs a slight correction.
To explain this, let us use the same notations as in \cite{Sz18}, Lemma 2.1.
The problem is that, in the proof, the condition $qa=q$ is not confirmed. (This does not follow from $q\in \overline{aAa}$.)
However this is easily fixed as follows.
Since $D$ is saturated, by standard applications of functional calculus,
one can find a positive element $a_0$ in $D$
with $a_0 a= a_0$, $\|a_0 \| =1$.
Applying the proof of \cite{Sz18}, Lemma 2.1 to $a_0$ instead of $a$,
we obtain the desired $q$ (for the original $a$).
\end{Rem}
We next show a variant of Lemmas 2.4 and 2.9 in \cite{Sz18}.
Because of the restrictions in Lemma \ref{Lem:fp},
we need to choose all objects more carefully.

\begin{Lem}\label{Lem:cob}
Let $u \colon G \rightarrow \mathcal{U}(\Cn{2})$
be a unitary representation.
Let $\alpha \colon G \curvearrowright \Cn{2}$
be an equivariantly $\Cn{2}$-absorbing, pointwise outer, QAP action.
Then there exist a unital $\ast$-homomorphism
$\nu \colon \Cn{2} \rightarrow ((\Cn{2})_\omega)^\alpha$ and $w\in \mathcal{U}( (\Cn{2})_\omega)$
satisfying $\nu(u_g)= w \alpha_g(w^\ast)$ for all $g\in G$.
\end{Lem}
\begin{proof}
Since a cocycle conjugation induces the conjugation
of the actions on the central sequence algebras,
we may assume that $\alpha$ is of the form
\[\id_{\Cn{2}} \otimes \alpha_0 \colon G \curvearrowright \Cn{2} \otimes \Cn{2}\] for some QAP action $\alpha_0$.
Put $A:=\Cn{2} \otimes 1_{\Cn{2}}$, $B := 1_{\Cn{2}}\otimes \Cn{2}$, $C:=\Cn{2} \otimes \Cn{2}$.
Take two unital $\ast$-homomorphisms
$\nu, \mu \colon \Cn{2} \rightarrow A_\omega$
satisfying $[\nu(x), \mu(y)]=0$ for all $x, y\in \Cn{2}$ (\cite{Ror}).
Regarding $\nu$ as a map into $(C_\omega)^\alpha$, we will show that this $\nu$ satisfies the desired condition.

Define
$\beta \colon G \curvearrowright \mathbb{M}_2(C^\omega)$
to be
\[\beta_g := \ad \left(
    \begin{array}{cc}
      1 & 0  \\
      0 & \nu(u_g) 
    \end{array}
  \right) \circ (\id_{\mathbb{M}_2} \otimes \alpha_g) \qquad{\rm~for~}g\in G.\]
Note that
$\beta$ is saturated by Corollary 1.10 and Proposition 1.12 in \cite{Sz18}.
Set \[D:=\mathbb{M}_2(C_\omega) \subset \mathbb{M}_2(C^\omega).\]
Since $\mu(\Cn{2}) \subset (C_\omega)^\alpha \cap \nu(\Cn{2})'$, we have a unital inclusion
\[1_{\mathbb{M}_2} \otimes \mu(\Cn{2}) \subset D^\beta \cap \{e_{1, 1}, e_{2, 2}\}'.\]
This shows
\[[e_{1, 1}]_0 = [e_{2, 2}]_0 =0 \qquad{\rm~in~}K_0(D^\beta).\]

We next show that, with respect to $\beta$,
the inclusion $1_{\mathbb{M}_2} \otimes C \subset \mathbb{M}_2(C^\omega)$
fulfills the assumptions of Lemma \ref{Lem:fp}.
Since $D= \mathbb{M}_2(C^\omega) \cap(1_{\mathbb{M}_2} \otimes C)' $,
this claim together with Lemma \ref{Lem:fp}
shows that $D^\beta$ is purely infinite simple.
Condition (1) follows from Proposition 3.4 of \cite{KP}.
Condition (2) follows from
Theorem 3.1 of \cite{Sz18} (cf.~\cite{Na}).
To show condition (3), fix finite subsets $T \subset  1_{\mathbb{M}_2}\otimes B$,
 $F \subset G$ and a positive number $\epsilon>0$.
Then, as $\alpha_0$ has the QAP,
one can take $\xi \in \ell^2(G, 1_{\mathbb{M}_2} \otimes B) \subset \ell^2(G, \mathbb{M}_2(C^\omega))$
satisfying
\begin{itemize}
\item $\langle \xi, (\widetilde{\id_{\mathbb{M}_2} \otimes \alpha})_g(\xi) \rangle \approx _\epsilon 1$ for all $g\in F$,
\item $b\xi \approx_\epsilon \xi b$ for all $b\in T$.
\end{itemize}
Observe that, for any $g\in G$, as the values of $\xi$ commute with $\left(
    \begin{array}{cc}
      1 & 0  \\
      0 & \nu(u_g) 
    \end{array}
  \right)\in \mathbb{M}_2(A_\omega)$, we have
$(\widetilde{\id_{\mathbb{M}_2} \otimes \alpha})_g(\xi) =\tilde{\beta}_g(\xi)$.
Hence $\langle \xi, \tilde{\beta}_g(\xi) \rangle \approx _\epsilon 1$ for all $g\in F$.
We also note that $a \xi = \xi a$ for all $a\in 1_{\mathbb{M}_2} \otimes A$.
As $(1_{\mathbb{M}_2} \otimes A) \cup( 1_{\mathbb{M}_2} \otimes B)$ generates $1_{\mathbb{M}_2} \otimes C$ as a \Cs-algebra, this proves condition (3) of the inclusion.

Hence, by Lemma \ref{Lem:fp}, $D^\beta$ is purely infinite simple.
Therefore, by \cite{Cun}, one can take a partial isometry
$v\in D^\beta$ satisfying
$vv^\ast =e_{1, 1}, v^\ast v =e_{2, 2}$.
Certainly this element is of the form
$v= e_{1, 2} \otimes w$
for some $w\in \mathcal{U}(C_\omega)$.
Since $v$ is $\beta$-invariant,
direct computations (cf.~\cite{Sz18}, Lemma 2.4) show that
$\nu(u_g)= w \alpha_g(w)^\ast$ for all $g\in G$.
We thus obtain the desired element $w$.
\end{proof}

\begin{Cor}\label{Cor:emb}
Let $\alpha \colon G \curvearrowright \Cn{2}$
be an equivariantly $\Cn{2}$-absorbing, pointwise outer, QAP action.
Then for any unitary representation $u \colon G \rightarrow \mathcal{U}(\Cn{2})$,
there is a unital $G$-equivariant embedding
$(\Cn{2}, \ad(u)) \rightarrow ((\Cn{2})_\omega, \alpha_\omega)$.

\end{Cor}
\begin{proof}
By Lemma \ref{Lem:cob}, we have a unital $\ast$-homomorphism $\nu \colon \Cn{2} \rightarrow ((\Cn{2})_\omega)^\alpha$ and $w\in \mathcal{U}((\Cn{2})_\omega)$
satisfying
$w \alpha_g(w)^\ast= \nu(u_g)$ for all $g\in G$.
Set $\mu:= \ad(w^\ast)\circ \nu$.
Then direct computations (cf.~\cite{Sz18}, Lemma 2.9) show that
$\mu$ gives the desired embedding. 
\end{proof}
We now prove Theorem \ref{Thmint:QAP}.
Before the proof, we record the following universality-like property of QAP actions, which is of independent interest.
\begin{Prop}\label{Prop:QAP}
Let $\alpha \colon G \curvearrowright A$ be a QAP action on a unital simple separable nuclear \Cs-algebra.
Then for any action $\beta \colon G \curvearrowright B$ on a unital separable exact \Cs-algebra,
there is a $G$-equivariant unital completely positive map
from $B$ into $A_\omega$.
Moreover, when we additionally assume that $\alpha$ is pointwise outer and equivariantly $\Cn{2}$-absorbing,
the stated maps can be taken to be an embedding.
\end{Prop}
\begin{proof}
By the observation in the first paragraph of the proof of Proposition 5.3 in \cite{Sz18},
we only need to consider the case $(B, \beta)= (\Cn{2}, \ad(u))$
for unitary representations $u\colon G \rightarrow \mathcal{U}(\Cn{2})$.
(In \cite{Sz18} the statement is restricted to amenable groups, but the same proof works for exact groups after employing the reduced \Cs-completion instead of the full one; see ~\cite{BO}, Theorem 10.2.9.)

The last statement follows from this observation and Corollary \ref{Cor:emb}

Let $u\colon G \rightarrow \mathcal{U}(\Cn{2})$ be a unitary representation.
Take an equivariantly $\Cn{2}$-absorbing, pointwise outer action
$\sigma \colon G \curvearrowright \Cn{2}$ which admits an invariant state $\varphi$.
(For instance, take the diagonal action of the
trivial action on $\Cn{2}$ and the Bernoulli shift action $G \curvearrowright \bigotimes_G \Cn{2}$.)
Then, since $A \otimes \Cn{2} \cong \Cn{2}$, the diagonal action $\alpha \otimes \sigma$
satisfies the assumptions of Corollary \ref{Cor:emb}.
Therefore we have
a unital $G$-equivariant embedding
\[\iota \colon (\Cn{2}, \ad(u)) \rightarrow ((A \otimes \Cn{2})_\omega, \alpha \otimes \sigma).\]
As $\varphi$ is $\sigma$-invariant, the map
$\id_{A} \otimes \varphi \colon A \otimes \Cn{2} \rightarrow A$
induces a $G$-equivariant unital completely positive map
\[\Phi \colon ((A \otimes \Cn{2})_\omega, \alpha \otimes \sigma) \rightarrow (A_\omega, \alpha);\quad [x_n]_n^\omega \mapsto [(\id_A \otimes \varphi)(x_n)]_n^\omega.\]
The composite $\Phi \circ \iota$ gives the desired map.
\end{proof}
\begin{proof}[Proof of Theorem \ref{Thmint:QAP}]
Clearly (6) implies (4) and (5). 
By Proposition 2.5 of \cite{AD}, (3) implies (2).
We show that (1) implies (3),  (2) implies (6), (4) implies (1), and (5) implies (1), which complete the proof.

(1)$\Rightarrow$ (3):
Assume that $\alpha$ has the QAP.
Take an amenable action $\eta \colon G \curvearrowright X$
on a compact metrizable space (\cite{Oz}, \cite{BO}, Theorem 5.1.7).
By Proposition \ref{Prop:QAP},
we have a $G$-equivariant unital completely positive
map $\Phi \colon C(X) \rightarrow A_\omega$.
Since $\eta$ is amenable, by \cite{BO}, Theorem 4.4.3,
there is a sequence $(f_n \colon G \rightarrow C(X))$
of finitely supported positive definite functions
with $\lim_{n\rightarrow \infty} f_n(g) =1$ for all $g\in G$.
Since $\Phi$ is completely positive and $G$-equivariant,
the functions $\Phi \circ f_n \colon G \rightarrow A_\omega$
are positive definite.
As $\Phi$ is unital and continuous, we have 
\[\lim_{n \rightarrow \infty}(\Phi \circ f_n)(g)=1 \qquad{\rm~ for~ all~} g\in G.\]
This confirms condition (3) of $\alpha$.

(2)$\Rightarrow$ (6):
Assume that $\alpha$ satisfies condition (2).
Let $B \subset A^\omega$ be a separable $G$-\Cs-subalgebra.
Take any finite subset $F\subset G$ and finite sequence $a(1), \ldots, a(k) \in A^\omega$.
Fix a positive number $\epsilon>0$.
Choose a dense sequence $(b(n))_{n=1}^\infty$ of $B$.

For each $i=1, \ldots, k$, take a representing sequence $(a(i)_n)_{n\in \mathbb{N}}$
of $a(i)$. 
We also fix a  representing sequence $(b(i)_n)_{n\in \mathbb{N}}$
of $b(i)$ for each $i\in \mathbb{N}$. 
Choose $\xi  \in \ell^2(G, A_\omega)$
satisfying $\langle \xi, \tilde{\alpha}_g(\xi)\rangle \approx_{\epsilon} 1$ for all $g\in F$.
By a standard perturbation argument, we may assume in addition that $\xi$
is finitely supported.
Choose a bounded sequence $(\xi_{n})_{n=1}^\infty$ in  $\ell^2(G, A)$
satisfying
\begin{itemize}
\item for any $n \in \mathbb{N}$, $\supp(\xi_{n}) \subset \supp(\xi)$,
\item for each $g\in G$, $[\xi_{n, g}]_n^\omega = \xi_{g}$.
\end{itemize}
Then $\lim_{n\rightarrow \omega} \| \langle \xi_{n}, \tilde{\alpha}_g(\xi_{n}) \rangle -1\|< \epsilon$ for all $g\in F$, $\lim_{n \rightarrow \omega} \| a\xi_{n} - \xi_{n} a\| =0$ for all $a\in A$.
By these two conditions, for each $n\in \mathbb{N}$, one can choose $k(n)\in \mathbb{N}$
satisfying $\langle \xi_{k(n)}, \tilde{\alpha}_g(\xi_{k(n)}) \rangle \approx_{\epsilon} 1$ for all $g\in F$
and $\| x \xi_{k(n)} - \xi_{k(n)} x\| < 1/n$ for all $x= a(1)_n, \ldots, a(k)_n, b(1)_n, \ldots, b(n)_n$.
Define $\eta \in \ell^2(G, A^\omega)$ to be
$\eta_g:= [\xi_{k(n), g}]_n^\omega$ for $g\in G$.
Then by the choice of $k(n)$'s, we obtain
\begin{itemize}
\item
$\langle \eta, \tilde{\alpha}_g(\eta)\rangle \approx_{\epsilon} 1$ for all $g\in F$,
\item  $a_i \eta = \eta a_i$ for all $i=1, \ldots, k$,
\item $\eta\in \ell^2(G, A^\omega \cap B')$.
\end{itemize} 
This proves condition (6).

(4)$\Rightarrow$ (1):
Assume that $\alpha^\omega$ has the QAP.
Then for any finite subsets $F \subset G$, $S\subset A$, and any $\epsilon>0$,
one can find an element $\xi \in \ell^2(G, A^\omega)$
with
\begin{itemize}
\item
$\langle \xi, \tilde{\alpha}_g(\xi)\rangle \approx_{\epsilon} 1$ for all $g\in F$,
\item
$\xi$ is finitely supported,
\item  $\|x\xi-\xi x\| < \epsilon$ for all $x\in S$.
\end{itemize} 
Choose a bounded sequence $(\xi_{n})_{n=1}^\infty$ in $\ell^2(G, A)$
with
\begin{itemize}
\item $\supp(\xi_n) \subset \supp(\xi)$ for all $n\in \mathbb{N}$,
\item $[\xi_{n, g}]_n^\omega = \xi_g$ for all $g\in G$.
\end{itemize}
Then by the choice of $\xi$, one can find $n \in \mathbb{N}$
satisfying
\begin{itemize}
\item
$\langle \xi_n, \tilde{\alpha}_g(\xi_n)\rangle \approx_{\epsilon} 1$ for all $g\in F$,

\item  $\|x\xi_n-\xi_n x\| < \epsilon$ for all $x\in S$.
\end{itemize}
This shows that $\alpha$ has the QAP.

(5)$\Rightarrow$ (1): The proof of this implication is similar to that of (4) $\Rightarrow$ (1). We therefore omit details.
\end{proof}

\begin{Rem}
By the proof of Theorem \ref{Thmint:QAP}
and obvious implications,
for any countable group action on a unital separable \Cs-algebra,
conditions (2) to (6) are pairwise equivalent,
and these conditions imply condition (1).
\end{Rem}

Now by Lemma \ref{Lem:fp} and Theorem \ref{Thmint:QAP},
we conclude the following useful result.
\begin{Cor}\label{Cor:pi}
Let $\alpha \colon G \curvearrowright A$
be a pointwise outer, QAP action on a unital Kirchberg algebra.
Then for any unital simple separable nuclear $G$-\Cs-subalgebra $B \subset A^\omega$,
the fixed point algebra $(A^\omega \cap B')^\alpha$ is purely infinite simple.
\end{Cor}
\begin{proof}
By Lemma \ref{Lem:fp},
it suffices to show that  the inclusion $B \subset A^\omega$ satisfies
the assumptions of Lemma \ref{Lem:fp}. 

Condition (1) follows from Theorem 1.17 of \cite{Sz18} (cf.~Proposition 3.4 of \cite{KP}). Condition (2) follows from Theorem 3.1 of \cite{Sz18} (cf.~\cite{Na})
and standard reindexation arguments (cf.~\cite{Sz18}, Lemmas 5.1 and 5.2).
By Theorem \ref{Thmint:QAP}, $\alpha^\omega$ has the QAP.
In particular, the inclusion satisfies condition (3).
\end{proof}

The rest of the proof of the Main Theorem is mostly identical to that
of Theorem C in \cite{Sz18}
modulo using Lemma \ref{Lem:fp}, Theorem \ref{Thmint:QAP}, and Corollary \ref{Cor:pi}
in place of \cite{Sz18}, Lemma 2.3.

\begin{Lem}[cf.~\cite{Sz18}, Lemma 5.1]\label{Lem:5.1}
Let $D$ be a unital purely infinite simple \Cs-algebra.
Let $\alpha \colon G \curvearrowright D$ be a saturated, QAP action.
Assume that any separable $G$-\Cs-subalgebra $A\subset D$
satisfies the following conditions.
\begin{itemize}
\item
There is a unital embedding $\Cn{2} \rightarrow (D \cap A')^\alpha$.
\item
The restriction action $G \curvearrowright D \cap A'$ is faithful.
\end{itemize}
Let $\beta \colon G \curvearrowright B$ be an action
on a unital simple separable nuclear \Cs-algebra $B$.
Then any two unital $G$-equivariant $\ast$-homomorphisms $(B, \beta) \rightarrow (D, \alpha)$
are $G$-unitarily equivalent.
\end{Lem}
\begin{proof}
The proof is the same as the original proof of \cite{Sz18}, Lemma 5.1,
after using Lemma \ref{Lem:fp} in place of \cite{Sz18}, Lemma 2.3.
\end{proof}
\begin{proof}[Proof of the Main Theorem]
We first observe that the statement in the first paragraph of Lemma 5.2 in \cite{Sz18}
holds true without amenability of the acting group. (This is an easy consequence of Theorem 3.1 in \cite{Sz18} (cf.~\cite{Na}) and standard reindexation arguments.)
Now the Main Theorem follows by the same proof as that of Theorem 5.5 of \cite{Sz18} (cf.~\cite{SzI}, Lemma 2.1),
after using Lemma \ref{Lem:5.1} and the second statement of Proposition \ref{Prop:QAP} instead of Lemma 5.1 and Corollary 5.4 in \cite{Sz18} respectively.
\end{proof}

\begin{Rem}\label{Rem:fp}
Here we remark that, for countable non-amenable exact groups $G$,
there is a pointwise outer, saturated action
of $G$ on a unital purely infinite simple \Cs-algebra
whose fixed point algebra fails to be simple.

To see this, choose an equivariantly $\Cn{2}$-absorbing, pointwise outer
action $\alpha \colon G \curvearrowright \Cn{2}$ which admits an invariant state.
(For instance, take the diagonal action of the trivial action on $\Cn{2}$
and the Bernoullli shift action $G \curvearrowright \bigotimes_{G} \Cn{2}$.)
Then notice that $\alpha_\omega$ also admits an invariant state.
Take a unitary representation
$u \colon G \rightarrow \mathcal{U}(\Cn{2})$
whose adjoint action $\ad(u)$ does not have an invariant state; for instance choose one of any $u_n$ from Example \ref{Exam:model}.

Then the statement of Corollary \ref{Cor:emb}, hence that of Lemma \ref{Lem:cob}, fails for these $\alpha$ and $u$.
Reviewing the proof of Lemma \ref{Lem:cob}, notice that the QAP of the considered action is used only to show
the simplicity of the fixed point algebra $D^\beta$ of
a pointwise outer, saturated action $\beta \colon G \curvearrowright D$
on a unital purely infinite simple \Cs-algebra (cf.~Lemma \ref{Lem:fp}).
Indeed, 
starting the argument in the proof to these $\alpha$ and $u$, (using the notations in the proof of Lemma \ref{Lem:cob},) one can still show that the (nonzero) projections $e_{1, 1}$ and $e_{2, 2}$ 
are properly infinite in $D^\beta$ and satisfy $[e_{1, 1}]_0=[e_{2, 2}]_0$ in $K_0(D^\beta)$.
Therefore, by \cite{Cun}, just from the simplicity of $D^\beta$,
one can conclude that these two projections are Murray--von Neumann equivalent in $D^\beta$.
As we have seen in the proof of Lemma \ref{Lem:cob},
this proves the statement of Lemma \ref{Lem:cob} for $\alpha$ and $u$.
This is a contradiction.
Thus $D^\beta$ is not simple.
\end{Rem}

\begin{Rem}\label{Rem:abs}
On the one hand,
similar to amenable groups (\cite{Izu04}, \cite{Sz18}, Remark 5.8), a pointwise outer, QAP
action of a non-amenable group on $\Cn{2}$ is not necessary equivariantly
$\Cn{2}$-absorbing.
To see this, recall from the proof of Theorem 5.1 in \cite{Suz19}
that the free group $\mathbb{F}_\infty$ admits a pointwise outer, QAP action $\alpha$
on a unital Kirchberg algebra $A$
whose reduced crossed product is stably isomorphic to $\Cn{\infty}$.
As observed in \cite{Izu04},
there is a pointwise outer action
$\beta \colon \mathbb{Z}_2 \curvearrowright \Cn{2}$
whose reduced crossed product has nonzero K-theory.
Then the product action $\zeta \colon \mathbb{F}_\infty \times \mathbb{Z}_2 \curvearrowright A \otimes \Cn{2} \cong \Cn{2}$
is pointwise outer, QAP, but its reduced crossed product has nonzero K-theory.
Thus $\zeta$ is not equivariantly $\Cn{2}$-absorbing.
On the other hand,
this phenomenon is expected to be related to torsions of the acting groups.
Works on amenable group actions (see e.g., \cite{IM2}, \cite{Mat}, \cite{Sz18b}) suggest that the equivariant $\Cn{2}$-absorbing assumption could be removed from the Main Theorem
for certain torsion free groups.
\end{Rem}
\section{Some equivariant $\Cn{\infty}$-absorption results for
exact groups}\label{Sec:Oinf}
In this section we prove Theorem \ref{Thmint:Oinf}, an equivariant variant of the Kirchberg $\Cn{\infty}$-absorption
theorem for QAP actions. This is an extension of another main result in \cite{Sz18} (Theorems B, 3.4, and Corollary 3.7) obtained for amenable groups.
However we remark that Theorem \ref{Thmint:Oinf} is not as satisfactory as the Main Theorem for non-amenable groups
because the absorbed model actions themselves do not have the QAP.
That said, we also expect that this result gives a new insight of
amenable actions on Kirchberg algebras.

\begin{proof}[Proof of Theorem \ref{Thmint:Oinf}]
We first observe that it suffices to show the following statement: For any unitary representation
$u \colon G \rightarrow \mathcal{U}(\Cn{\infty})$ which factors through $\Cn{2}$ (see Definition \ref{Defint:fact}),
there is a unital $G$-equivariant
$\ast$-homomorphism
$(\Cn{\infty}, \ad(u)) \rightarrow (A_\omega, \alpha)$.
Indeed once we have confirmed the claim,
standard reindexation arguments guarantee the existence
of a unital $G$-equivariant $\ast$-homomorphism
$\bigotimes_{\mathbb{N}}(\Cn{\infty}, \gamma) \rightarrow (A_\omega, \alpha)$.
By Proposition 5.3 of \cite{SzII}, $\bigotimes_{\mathbb{N}}\gamma$ is strongly self-absorbing.
Therefore Theorem 3.7 of \cite{SzI} or Theorem 4.11 of \cite{IM} completes the proof.

Let $u \colon G \rightarrow \mathcal{U}(\Cn{\infty})$
be a unitary representation factoring through $\Cn{2}$.
Take a unitary representation $v\colon G \rightarrow \U(\Cn{2})$
and an embedding $\iota \colon \Cn{2} \rightarrow \Cn{\infty}$
satisfying $u_g = \iota(v_g)+(1_{\Cn{\infty}}-\iota(1_{\Cn{2}}))$ for all $g\in G$.

By Corollary \ref{Cor:pi}, $(A_\omega)^\alpha$ is purely infinite simple.
In particular one can find a unital $\ast$-homomorphism
$\mu_0 \colon \Cn{\infty} \rightarrow (A_\omega)^\alpha$.
Set $p:= \mu_0(\iota(1_{\Cn{2}})) \in  (A_\omega)^\alpha$.
Set $D:= p A_\omega p$. Note that $A_\omega$, thus $D$, is purely infinite simple (\cite{KP}, Proposition 3.4).
Denote by $\zeta \colon G \curvearrowright D$ the restriction action of $\alpha_\omega$.
Then, as $D^\zeta = p(A_\omega)^\alpha p$, $D^\zeta$ is purely infinite simple.
Note that $\mu_0( \iota(\Cn{2})) \subset D^\zeta$.
By Theorem \ref{Thmint:QAP}, $\alpha_\omega$, therefore $\zeta$, has the QAP.

Define $\beta \colon G \curvearrowright \mathbb{M}_2(D)$ to be
\[\beta_g:=\ad \left(
    \begin{array}{cc}
      p & 0  \\
      0 & \mu_0(u_g)p 
    \end{array}
  \right) \circ (\id_{\mathbb{M}_2} \otimes \zeta_g) \qquad {\rm~for~}g\in G.\]
Observe that, as $\zeta$ has the QAP, so does $\beta$.
By Corollary 1.10 and Proposition 1.12 in \cite{Sz18}, $\beta$ is saturated.
Since $\mu_0(u_g)p=\mu_0(\iota(v_g))$, one can find a sequence $(\nu_n \colon \Cn{2} \rightarrow \mu_0(\iota(\Cn{2})))_{n=1}^\infty$
of unital $\ast$-homomorphisms
with $\lim_{n\rightarrow \infty} [\nu_n(x), \mu_0(u_g)p ] =0$ for all $g\in G$.
Since $\beta$ is saturated,
one can take a unital embedding
\[\Cn{2} \rightarrow \mathbb{M}_2(D)^\beta \cap \{e_{1, 1}, e_{2, 2}\}'.\]
Hence \[[e_{1, 1}]_0 =[e_{2, 2}]_0=0 \qquad
{\rm in~} K_0(\mathbb{M}_2(D)^\beta).\]
Since $\beta$ is saturated and satisfies the QAP, 
Lemma \ref{Lem:fp} yields that  $\mathbb{M}_2(D)^\beta$ is purely infinite simple.
Thus, by \cite{Cun}, one can find a unitary $w_0\in \U(D)$
satisfying $e_{1, 2} \otimes w_0\in  \mathbb{M}_2(D)^\beta$.
Direct computations (cf.~ \cite{Sz18}, Lemma 2.9) show that
$\mu_0(\iota(v_g))=w_0\alpha_g(w_0^\ast)$ for all $g\in G$.
Set $w:= w_0 + (1-p) \in \mathcal{U}(A_\omega)$.
Then $\mu_0(u_g) = w \alpha_g(w^\ast)$.

Now define $\mu:= \ad(w^\ast) \circ \mu_0 \colon \Cn{\infty} \rightarrow A_\omega$.
Then direct computations show that
\[\mu \circ \ad(u_g) = \alpha_g \circ \mu \qquad {\rm~for~ all~} g\in  G.\]
As explained in the first paragraph of the proof, this completes the proof.
\end{proof}
\subsection*{Concluding remarks}
The Main Theorem concludes the existence of \emph{the largest} action on Kirchberg algebras for exact groups (with respect to the tensor products, up to strong cocycle conjugacy).
For $\Cn{\infty}$, certainly it is more desirable to obtain an equivariant $\Cn{\infty}$-absorption theorem for some model action \emph{with the QAP} instead of our present $\gamma$.
More precisely, to understand (pointwise outer) QAP actions (on Kirchberg algebras) more deeply, it is desirable to find \emph{the smallest} QAP actions on Kirchberg algebras in the same sense.
However we remark that
even the existence of a QAP action on $\Cn{\infty}$
is not clear.
At the moment we know such actions only for free groups 
(and therefore for some groups constructed from free groups);
see the proof of Theorem 5.1 in \cite{Suz19}.
\subsection*{Acknowledgements}
The author is grateful to the referee for his/her careful reading and helpful suggestions.
This work was supported by JSPS KAKENHI Early-Career Scientists
(No.~19K14550).


\begin{thebibliography}{99}
\bibitem{AD79} C.~ Anantharaman-Delaroche, {\it Action moyennable d'un groupe localement compact sur une alg\`ebre de von Neumann.} Math.~ Scand.~ {\bf 45} (1979), 289--304.
\bibitem{AD}C.~Anantharaman-Delaroche, {\it Syst\`{e}mes dynamiques non commutatifs et moyennabilit\'e}, Math. Ann. 279 (1987), 297--315.
\bibitem{AD02} C.~Anantharaman-Delaroche, {\it Amenability and exactness for dynamical systems and their \Cs-algebras.} Trans.~Amer.~Math.~Soc. {\bf 354} (2002), no. 10, 4153--4178.
\bibitem{BO} N.~ P.~ Brown, N.~Ozawa, {\it \Cs-algebras and finite-dimensional approximations.} Graduate Studies in Mathematics {\bf 88}. American Mathematical Society, Providence, RI, 2008.
\bibitem{BEW} A.~Buss, S.~ Echterhoff, R.~Willett,
{\it Injectivity, crossed products, and amenable group actions.}
Contemporary~ Math.~{\bf 749} (2020), 105--138.
\bibitem{Con}A.~ Connes, {\it Periodic automorphisms of the hyperfinite factors of type II$_1$.} Acta Sci.~ Math.~ {\bf 39} (1977), 39--66.
\bibitem{Cun}J.~Cuntz, {\it K-theory for certain \Cs-algebras.} Ann.~of Math.~ {\bf 113} (1981), 181--197.
\bibitem{FH} I.~ Farah, B.~ Hart, {\it Countable saturation of corona algebras.} C.~ R.~ Math.~ Rep.~ Acad.~ Sci.~ Canada {\bf 35} (2013), no.~ 2, 35--56.
\bibitem{Gro}M.~ Gromov, {\it Random walk in random groups.} Geom.~ Funct.~ Anal.~ {\bf 13} (2003), no.~1, 73--146.
\bibitem{Izu04}M.~ Izumi, {\it Finite group actions on \Cs-algebras with the Rohlin property I.} Duke Math.~ J.~ {\bf 122} (2004), no.~2, 233--280.
\bibitem{IICM} M.~ Izumi, {\it Group Actions on Operator Algebras.} Proc.~ Intern.~ Congr.~ Math.~ (2010), 1528--1548.
\bibitem{IM}M.~Izumi, H.~Matui, {\it Poly-$\mathbb{Z}$ group actions on Kirchberg algebras I.}
To appear in Int.~ Math.~ Res.~ Not., arXiv:1810.05850.
\bibitem{IM2}M.~Izumi, H.~Matui, {\it Poly-$\mathbb{Z}$ group actions on Kirchberg algebras II.}
To appear in Invent.~Math., arXiv:1906.03818v2.
\bibitem{Kir} E.~ Kirchberg, {\it The classification of purely infinite \Cs-algebras using Kasparov's theory.} Preprint.
\bibitem{KP}E.~ Kirchberg, N.~ C.~ Phillips, {\it Embedding of exact \Cs-algebras in the Cuntz algebra $\Cn{2}$.} J.~ reine angew.~ Math.~ {\bf 525} (2000), 17--53.
\bibitem{KW}E.~ Kirchberg, S.~ Wassermann, {\it Exact groups and continuous bundles of \Cs-algebras.} Math.~ Ann.~ {\bf 315} (1999), no.2, 169--203.
\bibitem{Mat}H.~Matui, {\it Classification of outer actions of $\mathbb{Z}^N$ on $\Cn{2}$.}
Adv.~ Math.~ {\bf 217} (2008), 2872--2896.
\bibitem{Na}H.~Nakamura, {\it Aperiodic automorphisms of nuclear purely infinite simple \Cs-algebras.} Ergodic Theory Dynam. Systems {\bf 20} (2000), 1749--1765. 
\bibitem{Os} D.~Osajda, {\it Small cancellation labellings of some infinite graphs and applications.} Acta Math.~ {\bf 225} (2020), no. 1, 159--191.
\bibitem{Oz}N.~ Ozawa, {\it Amenable actions and exactness for discrete groups.}
C.~ R.~ Acad.~ Sci.~ Paris Ser.~ I Math., {\bf 330} (8) (2000), 691--695.
 \bibitem{OzICM}N.~ Ozawa, {\it Amenable Actions And Applications.}
International Congress of Mathematicians, Vol.~ II, 1563--1580.
\bibitem{Phi}N.~ C.~ Phillips, {\it A classification theorem for nuclear purely infinite simple \Cs-algebras.} Doc.~ Math.~ {\bf 5} (2000), 49--114.

\bibitem{Ror}M.~ R{\o}rdam, {\it Classification of nuclear, simple \Cs-algebras.} vol.~126 of Encyclopaedia Math.~ Sci., Springer, Berlin, 2002, 1--145.
\bibitem{Suzeq}Y.~Suzuki, {\it Simple equivariant \Cs-algebras whose full and reduced crossed products coincide.}
J.~ Noncommut.~ Geom. {\bf 13} (2019), 1577--1585.
\bibitem{Suz19}Y.~Suzuki, {\it Complete descriptions of intermediate operator algebras by intermediate extensions of dynamical systems.}
Comm.~Math.~Phys. {\bf 375} (2020), 1273--1297.
\bibitem{Suz20a}Y.~Suzuki, {\it On pathological properties of fixed point algebras in Kirchberg algebras.} Proc.~ Roy.~ Soc.~ Edinburgh Sect.~ A {\bf 150} (2020), 3087--3096.
\bibitem{Suz20b}Y.~Suzuki, {\it Non-amenable tight squeezes by Kirchberg algebras.}
Preprint, arXiv:1908.02971.
\bibitem{SzI}G.~ Szab\'o, {\it Strongly self-absorbing C*-dynamical systems.} Trans. Amer. Math. Soc. {\bf 370} (2018), 99--130 [with an eratta].
\bibitem{SzII}G.~ Szab\'o, {\it Strongly self-absorbing C*-dynamical systems II.} J.~ Noncomm.~ Geom.~ {\bf 12} (2018), no.~ 1, 369--406.
\bibitem{Sz18}G.~ Szab\'o, {\it Equivariant Kirchberg-Phillips-type absorption for amenable group actions.} Comm.~ Math.~ Phys., {\bf 361} (2018), no. 3, 1115--1154.
\bibitem{Sz18b}G.~ Szab\'o, {\it Actions of certain torsion-free elementary amenable groups on strongly self-absorbing \Cs-algebras.} Comm.~ Math.~ Phys.~ {\bf 371} (2019), no. 1, pp. 267--284.
\bibitem{TW}A.~ S.~ Toms, W.~ Winter, {\it Strongly self-absorbing \Cs-algebras.} Trans.~ Amer.~ Math.~ Soc.~ {\bf 359} (2007), no. 8, 3999--4029.
\end{thebibliography}
\end{document}